\documentclass[11pt]{amsart}
\usepackage{geometry}
\usepackage[english]{babel}
\usepackage{amsthm,amsmath,amsfonts,amssymb,commath}   %pacotes da American Mathematical Society
\usepackage{mathrsfs}
\usepackage{upgreek}
\usepackage{thmtools}
\usepackage{lipsum}
\usepackage{placeins}
\usepackage{babel}
 \usepackage{hyperref}
 \usepackage{setspace}
\usepackage{enumerate,multicol,stmaryrd}
\usepackage{tabularx}
 \usepackage{graphicx}
 \graphicspath{ {./images/} }
 \usepackage{pdfpages}
\usepackage{enumitem}
\setlist[itemize]{align=left,leftmargin=*,itemindent=0cm,labelindent=\parindentt,labelsep=2mm}
\setlist[enumerate]{label={\bfseries\alph*)},align=left,leftmargin=*,itemindent=0cm,labelindent=\parindent,labelsep=2mm}

\newdimen\parindentt
\parindentt=\parindent
\advance\parindentt by 5mm
\usepackage{indentfirst}   % Indentação do primeiro parágrafo

\newtheorem{theorem}{Theorem}[section]

\newtheorem{proposition}[theorem]{Proposition}
\newtheorem{definition}{Definition}
\newtheorem{corollary}[theorem]{Corollary}
\newtheorem{lemma}[theorem]{Lemma}
\newtheorem*{remark}{Remark}
\theoremstyle{definition}
\newtheorem{example}{Example}

\DeclareMathOperator{\ga}{\textsl{g}}

\def\rm#1{\mathrm{#1}}
\def\cal#1{\mathcal{#1}}
\def\bb#1{\mathbb{#1}}
\def\lie#1{\mathfrak{#1}}
\def\lr#1{\left\langle #1\right\rangle}
\usepackage[all]{xy}

\def\co{\colon}

\newcommand{\h}{\frac{1}{2}}

\title[A curvature approach to fatness]{A curvature approach to fatness}
\author{Leonardo F. Cavenaghi}
\address{Instituto de Matemática, Estatística e Computação Científica -- Unicamp, Rua Sérgio Buarque de Holanda, 651, 13083-859, Campinas, SP, Brazil}
\email{leonardofcavenaghi@gmail.com}

\author{Lino Grama}
\address{Instituto de Matemática, Estatística e Computação Científica -- Unicamp, Rua Sérgio Buarque de Holanda, 651, 13083-859, Campinas, SP, Brazil}
\email{lino@ime.unicamp.br}

\begin{document}

	\begin{abstract} 
This paper delves into the concept of ``fat bundles'' within Riemannian submersions. One explores the structural implications of fat Riemannian submersions, particularly focusing on those with non-negative sectional curvature. The main results include the classification of fibers as symmetric spaces, the isometric correspondence of fat foliations with coset foliations on Lie groups, and the rigidity of dual foliations associated with fat Riemannian submersions.
\end{abstract}
	
	\maketitle

 \pretolerance=1000
	\section{Introduction}
Let $\pi: F\hookrightarrow (M,\ga) \rightarrow (B,\ga_B)$ be a Riemannian submersion, where $F$ represents the fiber and $B$ is the base. Denote the vertical bundle of $\pi$ as $\mathcal{V}$, containing vectors tangent to the fibers. We term the horizontal bundle its $\ga$-orthogonal complementary bundle, which is isometric to $(TB,\ga_B)$. We say that the Riemannian submersion $\pi$ is \emph{fat} if, for every point $x$ in $M$ and every nonzero vector $X$ in $\mathcal{H}_x$, the following condition holds for a local horizontal extension $\widetilde{X}$ of $X$:
\begin{equation}\label{eq:fat}
[\widetilde{X},\mathcal{H}_x]^{\mathbf{v}} = \mathcal{V}_x.
\end{equation}
The left-hand side in Equation \eqref{eq:fat} is $2A_X\mathcal{H}_x$, where $A:\mathcal{H}_x\times \mathcal{H}_x\rightarrow \mathcal{V}_x$ stands for the O'Neill tensor of the submersion $\pi$. The superscript $\mathrm{v}$ represents the projection into the vertical bundle.

As stated, the condition of ``fatness'' (\cite{weinstein1980fat}) is independent of the Riemannian metric on $M$ and pertains solely to the submersion; it relies on the choice of the horizontal distribution. Nevertheless, it becomes particularly significant when studying Riemannian submersions with totally geodesic fibers. Pick $X\in \cal H_x$ and let $A_X^*$ be the $\ga$-dual of $A_X$. According to Gray \cite{gray1967pseudo} or O'Neill \cite{oneill}, for any nontrivial plane $X\wedge V$, where $X\in \mathcal{H}$ and $V\in \mathcal{V}$, the unreduced sectional curvature of $\ga$ at $X\wedge V$ is given by:
\begin{equation}\label{eq:vertical}
K_{\ga}(X,V) = |A^*_XV|_{\ga}^2.
\end{equation}
The condition of fatness can be translated as
\begin{definition}\label{def:verticalfat}
A Riemannian submersion $\pi: F\hookrightarrow (M,\ga)\rightarrow (B,\ga_B)$ with totally geodesic fibers $F$ is termed \emph{fat} if every non-degenerate vertizontal plane $X\wedge V$ has positive curvature.
\end{definition}
It turns out that Definition \ref{def:verticalfat} and the characterization expressed in Equation \eqref{eq:fat} are equivalent -- Proposition \ref{prop:equivalences}.

\bigskip

The fatness condition significantly restricts the possible dimensions of $\mathcal{V}$ compared to those of $\mathcal{H}$. Specifically, for a submersion to be fat, it must satisfy the condition that, for any point $x$ in $M$, $\dim \mathcal{V}_x \leq \dim \mathcal{H}_x - 1$. Moreover, if $\dim \mathcal{V}_x = \dim \mathcal{H}_x - 1$, then $\dim \mathcal{H}_x = 2, 4, 8$, among other cases (see either \cite[Proposition 2.5, p. 8]{Ziller_fatnessrevisited} or Proposition \ref{prop:dimcons}). Some structural results for fat Riemannian submersions were already known. However, the scarcity of examples of fat Riemannian submersions indicates that either the known construction techniques are insufficient or they constitute very particular examples; for a complete account, see \cite{Ziller_fatnessrevisited}.

This paper provides new structure results for fat Riemannian submersions with curvature assumptions. Our main focus are Riemannian submersions $\pi: F\hookrightarrow (M,\ga)\rightarrow (B,\ga_B)$ where $\ga$ has non-negative sectional curvature. We use techniques of the sub-field of ``positive curvatures'' to our results, biased by the concept of ``Dual foliations'' introduced in \cite{wilkilng-dual} and structure results from that coming. We summarize our results next.
\begin{itemize}
\item We show that if the fiber $F$ of a fat Riemannian submersion $\pi$ from a compact non-negatively curved manifold has dimension greater than $1$, then $F$ is a rank-one symmetric space -- Theorem \ref{thm:symmetricfibers}. This result is evidence that examples of fat Riemannian submersions are structurally restriced. It implies Theorem \ref{ithm:symmetricfiber}, classifying all possible holonomy groups for $\pi$ and possible fibers.
\item We show that every fat foliation $\cal F$ induced by the connected component fibers of submersions on compact Lie groups $G$ is isometric to a coset foliation induced by a subgroup $H<G$  -- Theorem \ref{thm:zillerq}. This result relates to Grove's conjecture on Riemannian foliations on Lie groups with bi-invariant metrics: \emph{Let $G$ be a compact simple Lie group with a bi-invariant metric. A Riemannian submersion $\pi:G\rightarrow B$ with connected totally geodesic fibers is induced either by left or right cosets}. To its proof, we combine a result of Sperança \cite{speranca_grove}, the non-metric nature of the fatness condition, and the fact that every Riemannian metric $\ga$ in $G$ making $\cal F$ a Riemannian foliation is \emph{twisted} (Definition \ref{def:twisted}). 
\item We study the rigidity of dual foliations associated with fat Riemannian submersions -- Theorem \ref{cor:guijarao}.
\item Related to Problem 1 in \cite{Ziller_fatnessrevisited}, concerning the known examples $\mathrm{S}^3,~\mathrm{SO}(3)$-fat principal bundles being 3-Sasakian manifolds, we show that a locally symmetric $3$-Sasakian manifold is the round sphere up to the universal covering -- Theorem \ref{ithmq:ziller}.

\end{itemize}

Recent developments related to the subject of fat submersions, which worked as inspiration for our results, have been achieved in \cite{DV2022, cavenaghilinollohann2, finotwostep, finogueo, ovando1, ovando2, BOCHENSKI2016132, BOCHENSKI2016131, ivanov1}.

\subsection*{A few words on the proof techniques}
The following general spirit inspires the results in this paper. The fatness condition is metric-independent. Dual-leaf type arguments are used in \cite{speranca2017on,speranca_grove} to obstruct the vertical bundle of Riemannian manifolds with non-negative sectional curvature and/or positive vertizontal curvature. Such an obstruction allows us to construct a symmetric space structure on the fibers of fat Riemannian submersions under non-negative curvature hypotheses, obtaining Theorem \ref{thm:symmetricfibers}. Direct inspection allows us to classify the possible fiber type (Theorem \ref{ithm:symmetricfiber}). To Theorem \ref{thm:grove}, we use the fact that any submersion's induced fat foliation on a Lie group can be assumed to be Riemannian and of totally geodesic leaves according a bi-invariant metric; this is because any Riemannian metric making a fat foliation $\cal F$ to be a Riemannian foliation is twisted (Definition \ref{def:twisted}) -- Theorem \ref{cor:guijarao}. Sperança's main result in \cite{speranca_grove} ensures that $\cal F$ is isometric to a coset foliation. To explicitly compute the vertizontal curvature of the manifolds here considered, we use the fact that fat bundles are associated bundles and the curvature formulae presented in \cite{Cavenaghi2022}, which ensures such curvatures are the curvature of Cheeger deformed metrics on associated bundles.

 % \subsection*{This paper is organized as follows.} In Section \ref{fat:recall}, we recall the very basic but important structure results in fat Riemannian submersions, which shall play a relevant role in this paper. Next, in Section \ref{sec:prelicheeger}, we recall the concept of Cheeger deformations. This is needed to provide a treatable curvature formula for the sectional curvature of fat Riemannian submersions. In Section \ref{sec:liegroups}, we prove some new structure results for fat Riemannian foliations. This paper contains the Appendix \ref{ap:llohann} where we prove some Ambrose--Singer theorem for Riemannian foliations with bounded holonomy (Definition \ref{def:bhol}) that play a significant role in the proof of Theorems \ref{cor:guijarao} and \ref{thm:symmetricfibers} and Corollary \ref{cor:symmetricimposed}.

 \section{Preliminaries in fat bundles and Cheeger deformations}

 Section \ref{fat:recall} settles notation and recalls some already known structure results for fat Riemannian submersions. In Section \ref{sec:prelicheeger}, we recall the concept of Cheeger deformations on associated fiber bundles. A curvature formula is provided. Its usefulness relies on the fact that fat Riemannian submersions can be seen as associate bundles.
 
 \subsection{Preliminary aspects of fat submersions}
\label{fat:recall}
Let $\pi: F\hookrightarrow (M,\ga) \rightarrow (B,\ga_B)$ be an arbitrary Riemannian submersion. Throughout, decompose $TM = \cal H\oplus \cal V$ where $\cal V$ stands to the sub-bundle of $TM$ collecting pointwise vectors tangent to the fibers $F$. We call $\cal V$ the \emph{vertical bundle}. For each $x\in M$, the vector space $\cal V_x$ is termed \emph{vertical space at $x$}. The pointwise $\ga$-complementary space $\cal H_x$ to $\cal V_x$ is named \emph{horizontal space at $x$}. The collection of these vector subspaces generates the \emph{horizontal fiber bundle} $\cal H$ complementary to $\cal V$. For the sake of self-containment, let us add more details on \emph{Holonomy groups} for Riemannian submersions and recall the concept of the \emph{Holonomy Principal bundle}. 

Following \cite{Guijarro2007}, a result due to Ehresmann states that if the fibers \( F \) of the submersion \( \pi \) are compact, then \( \pi \) is a locally trivial fibration. This means that for any point \( b \in B \), there exists an open neighborhood \( U \) containing \( b \) such that \( \pi^{-1}(U) \cong U \times F \), where \( F \) represents the typical fiber of the submersion. 

Recall that a curve in \( M \) is termed \emph{horizontal} if it is tangent to the distribution \( \mathcal{H} \) at every point. Any closed curve in \( B \) that begins and ends at the same point \( b \in B \) induces a diffeomorphism of the fiber \( F = \pi^{-1}(b) \) over \( b \) by lifting the curve horizontally to points in \( M \). The holonomy group \( \mathrm{Hol}(b) \) at the point \( b \) consists of all such diffeomorphisms governed by an appropriate composition rule. This group is trivial in the case of a Riemannian product \( F \times B \rightarrow B \). However, in general, the holonomy group is not a Lie group. Nonetheless, for homogeneous submersions, the holonomy group at any point does form a Lie group and serves as the structure group of the bundle (see \cite[Theorem 2.2]{Guijarro2007}).

Let \( \pi_P: P \rightarrow B \) be a \( G \)-principal bundle. A principal \( G \)-connection on \( \pi_P \) is defined as a distribution \( \widetilde{\mathcal{H}} \) on \( P \) that is complementary to the kernel of the differential \( \mathrm{d}\pi_P \) and is invariant under the \( G \)-action. Specifically, it satisfies the condition \( g_{\ast}\widetilde{\mathcal{H}}_p = \widetilde{\mathcal{H}}_{gp} \), where \( p \in P \) and \( g_{\ast} \) denotes the derivative of \( g \in G \) viewed as a map \( g: P \rightarrow P \). The \emph{holonomy group of the connection at \( b \)} comprises the diffeomorphisms of the fiber over \( b \) that are obtained by lifting loops based at \( b \) to curves that are everywhere tangent to \( \widetilde{\mathcal{H}} \).

When considering a Riemannian submersion \( \pi \) with totally geodesic fibers, Theorem 1.4.1 in \cite{gw} indicates that \( \pi \) is indeed a fiber bundle. Theorem 2.2 in \cite{Guijarro2007} demonstrates that if the assumption of totally geodesic fibers \( F \) is relaxed in favor of considering \( F \) as a homogeneous space, \( \pi: (M, \ga) \rightarrow (B, \ga_B) \) can still be regarded as a fiber bundle. In such a case, we can always extract from it a principal bundle \( P \rightarrow B \) (for further details, refer to Section \ref{sec:prelicheeger}). 

Proposition \ref{prop:connection} states that the horizontal distribution of the submersion \( \pi \) naturally induces a connection \( \widetilde{\mathcal{H}} \) on the principal bundle \( P \rightarrow B \), which is compatible with the Riemannian submersion metric \( \ga \) on \( M \) in the sense that \( \mathcal H = \mathrm{d}\bar{\pi}(\widetilde{\mathcal{H}} \times \{0\}) \), where \( \bar{\pi}: P \times F \rightarrow M \) is the projection described by Equation \eqref{eq:projection}. Furthermore, Proposition \ref{prop:fatfiber} asserts that if \( \pi \) is \emph{fat} (Definition \ref{def:fatness}), then the holonomy group of this decoupled associated principal bundle coincides with the holonomy group of \( \pi \).

 Let $\Gamma(\cal H),~\Gamma(\cal V)$ stands to the $\mathrm{C}^{\infty}(M)$-module of vector fields taking values at $\cal H,~\cal V$, respectively. Denote by $A:\Gamma(\cal H)\times\Gamma(\cal H)\rightarrow \Gamma(\cal V)$ the O'Neill tensor $A_XY:=\h [X,Y]^{\mathbf{v}}$. 
 \begin{definition}\label{def:fatness}
     We say that the Riemannian submersion $\pi$ is \emph{fat} if, for any $x\in M$,
\[2A_{\widetilde X}\cal H_x=[\widetilde X,\cal H_x]^{\mathbf{v}} = \cal V_x\]
for any non-zero $X\in \cal H_x$ and any local extension $\widetilde X$ of $X$. 
 \end{definition}
 \begin{remark}[The geometric flavor of the fatness assumption]
     The definition of fatness (Definition \ref{def:fatness}) does not require the existence of a Riemannian submersion metric to the submersion $\pi: M\rightarrow B$. It is solely related to choosing a complementary distribution to $\cal V=\bigcup_x\mathrm{ker}~\mathrm{d}\pi_x$. Nevertheless, it acquires a more geometric flavor under the assumption of $\pi$ carrying a Riemannian submersion metric with totally geodesic fibers -- Proposition \ref{prop:equivalences}.
 \end{remark}

This subsection provides a concise overview of key results concerning fat submersions. The primary references for this topic are \cite{Ziller_fatnessrevisited} and \cite{gw}.
\begin{proposition}[Theorem 2.7.2, p.98 in \cite{gw}]\label{prop:connection}
    Let $\pi: F\hookrightarrow (M,\ga) \rightarrow (B,\ga_B)$ be a Riemannian submersion with totally geodesic fibers $F$. Then, $\pi$ is a fiber bundle, and $\ga$ is a connection metric.
\end{proposition}
As an important consequence of the former, it holds
\begin{proposition}[Proposition 2.6, p.9 in \cite{Ziller_fatnessrevisited}]\label{prop:fatfiber}
Let $\pi: F\hookrightarrow (M,\ga) \rightarrow (B,\ga_B)$ be a fat Riemannian submersion with totally geodesic fibers $F$. Then the holonomy group $H=\mathrm{Hol}(b)$ of $\pi$ at some point $b\in B$ acts transitively in $F$, making it a homogeneous manifold.  Consequently, there exists $K<H$ such that $F=H/K$ and
total space $M$ is diffeomorphic to $P\times_H(H/K) \cong P/K$, where $P$ is the $\pi$-associated holonomy
principal bundle.
\end{proposition}

\begin{definition}[The $A^*_X$-dual to the O'Neill tensor]
       Let $\pi: (M,\ga)\rightarrow (B,\ga_B)$ be a Riemannian submersion. Pick any $x\in M$ and fix a non-zero $X\in\cal H_x$. We denote by $A^*_X$ the $\ga$-dual to $A_X$,
        \[A^*_X : \cal V_x \rightarrow \cal H_x.\]
\end{definition}

\begin{proposition}[Proposition 2.4, p.8 in \cite{Ziller_fatnessrevisited}]\label{prop:equivalences}
    Let $\pi: (M,\ga)\rightarrow (B,\ga_B)$ be a Riemannian submersion with totally geodesic fibers. Then the following are equivalent 
    \begin{enumerate}
        \item  for any non-zero $X\in \cal H_x$, the $\ga$-dual $A^*_X$ to $A_X$,
        \[A^*_X : \cal V_x \rightarrow \cal H_x\]
        is an injective map.
        \item for each non-zero $V\in \cal V_x$ the rule $\cal H_x\times \cal H_x\ni (X,Y) \mapsto \ga(A_XY,V)$ defines a non-degenerate two form in $\cal H_x$.
        \item $\dim \cal V_x \leq \dim \cal H_x -1$. If the quality holds then $A_X : \cal H_x\cap \{X\}^{\perp}\rightarrow \cal V_x$ is an isomorphism.
        \item the vertizontal curvature $K_{\ga}(X,V)$ is everywhere positive for non-degenerate vertizontal planes $X\wedge V\neq 0,~X\in \cal H_x,~V\in \cal V_x,~\forall x\in M$.
    \end{enumerate}
\end{proposition}

Other dimensional constraints are presented under the condition of fatness.
\begin{proposition}[Proposition 2.5 in \cite{Ziller_fatnessrevisited}]\label{prop:dimcons}
     Let $\pi: M\rightarrow B$ be a fat submersion. The following dimensional constraints hold
     \begin{enumerate}
         \item $\dim B = \dim \cal H_x$ is even;
         \item $\dim \cal V_x = \dim \cal H_x -1$ implies that $\dim B = 2, 4, 8$;
         \item if $\dim \cal V_x \geq 2$ then $\dim B = 4k$, while if $\dim \cal V_x \geq 4$ then $\dim B = 8k$. 
     \end{enumerate}
\end{proposition}
\begin{remark}[What do we mean by a fat Riemannian submersion]\label{rem:fat}
  Throughout this manuscript, whenever considering a fat Riemannian submersion $\pi: F\hookrightarrow (M,\ga)\rightarrow (B,\ga_B)$, we mean a fiber bundle with a structure group being the holonomy group $H$ of the Riemannian submersion $\pi$ and totally geodesic fibers diffeomorphic to the homogeneous spaces $F=H/K$. The total space $M$ is the total space of an associated bundle $M\cong P\times_HF$ where $H\rightarrow P\rightarrow B$ is the principal holonomy bundle, i.e., with structure group $H$. The metric $\ga$ is a connection metric with positive vertizontal curvature.
\end{remark}

\subsection{Preliminary aspects of Cheeger deformations for associated bundles}
\label{sec:prelicheeger}

Cheeger first introduced the so-called ``Cheeger deformations'' in the 1970s as a tool to produce non-negatively curved metrics on manifolds obtained as quotients of manifolds with isometric actions. Let $(M,\ga)$ be a compact connected Riemannian manifold. Let $G$ be a compact connected Lie group of positive dimension acting by isometries on $(M,\ga)$. Taking $Q$ to be a bi-invariant metric on $G$, we consider the product manifold $M\times G$ with the product metric $\ga + t^{-1}Q$ for $t>0$.  

Notably $G$ defines an isometric action on $(M\times G,\ga+t^{-1}Q)$, which we denote as $\star$, defined by
\begin{equation}\label{eq:star}
r\star (m,g) := (rm,rg),~r\in G,~(m,g)\in M\times G.
\end{equation}
The quotient (orbit map) projection $\pi':(m,g)\rightarrow g^{-1}m$ defines a principal bundle which induces from  $\ga+t^{-1}Q$ a family of $G$-invariant Riemannian metrics (as $t$-varies) $\ga_t$ on $M$, termed Cheeger deformations of $\ga$.

Let $\mathfrak{g}_x$ denote the Lie algebra of $G_x$, the isotropy subgroup at $x$. We denote by  $\lie m_x$ the $Q$-orthogonal complement of $\lie g_x$. We observe that $\lie m_x$ is isomorphic to the tangent space to the orbit $Gx$ via \emph{action fields}. We term a vector in $T_xGx$  \emph{vertical}, so in analogy with Riemannian submersion, we denote $T_xGx:=\cal V_x$, and say that $\cal V_x$ is \emph{the vertical space at $x$}. Its $\ga$-orthogonal complement, denoted as $\cal H_x$, is named the \emph{horizontal space at $x$}. Any tangent vector $\overline X(x) \in T_xM$ can be uniquely decomposed as $\overline X(x) = X + U^{\ast}_x$, where $X$ is horizontal and $U\in \lie m_x$. In this manner, $U^*_x$ is the corresponding action vector (concerning $U$) at $x$. When there is no risk of confusion, we omit referring to $x$ everywhere.

It is typical to consider three symmetric and positive definite tensors associated with Cheeger deformations -- \cite{mutterz, Muter}
\begin{definition}\label{def:orbittensors}
	\begin{enumerate}
		\item The \emph{orbit tensor} at $x$ is the linear map $O : \lie m_x \to \lie m_x$ defined by
		\[\textsl g(U^{\ast},V^{\ast}) = Q(OU,V),\quad\forall U^{\ast}, V^{\ast} \in \mathcal{V}_x\]
		\item For each $t > 0$ the orbit tensor $O_t:\lie m_x\to \lie m_x$ is characterized by 
		\[\textsl g_t(U^{\ast},V^{\ast}) = Q(O_tU,V), \quad\forall U^{\ast}, V^{\ast} \in \mathcal{V}_x\]
		\item The \emph{metric tensor} $C_t:T_pM\to T_xM$ of $\textsl g_t$ is defined as
		\[\textsl g_t(\overline{X},\overline{Y}) = \textsl g(C_t\overline{X},\overline{Y}), \quad\forall \overline{X}, \overline{Y} \in T_xM\]
	\end{enumerate}
 \end{definition}
	\begin{proposition}[Proposition 1.1 in \cite{mutterz}]\label{propauxiliar}The tensors above satisfy:
		\begin{enumerate}
			\item \label{eq:pt} $O_t = (O^{-1} + t1)^{-1} = O(1 + tO)^{-1}$,
			\item If $\overline{X} = X + U^{\ast}$ then $C_t(\overline{X}) = X + ((1 + tO)^{-1}U)^{\ast}$.
		\end{enumerate}
	\end{proposition}

We move to the more general concept of Cheeger deformations presented in \cite{Cavenaghi2022}, applied to associated bundles. Our interest in this to-be-present concept relies on Proposition \ref{prop:fatfiber}. Consider a fiber bundle $F\hookrightarrow M\stackrel{\pi}{\to}B$ with a compact structure group $G$ and fiber $F$. If $G$ acts effectively on $F$, then $G$ is a structure group for $\pi$ if there is a choice of local trivializations $\{(U_i,\phi_i:\pi^{-1}(U_i) \to U_i\times F)\}$ 
	such that, for every $i,j$ with $U_i\cap U_j \neq \emptyset$ there is a continuous function $\varphi_{ij} : U_i\cap U_j \to G$ satisfying
	\begin{equation*}
	\phi_i\circ\phi_j^{-1}(b,f) = (b,\varphi_{ij}(b)f),
	\end{equation*}
	for all $b\in U_i\cap U_j$. 
The existence of the collection \(\{\varphi_{ij}\}\) enables us to construct a principal \(G\)-bundle over the base space \(B\). This bundle is defined as \(P = \sqcup U_i \times G / \sim\), where the equivalence relation \(\sim\) is specified as follows: \((b, f) \sim (b', f')\) if and only if \(b = b'\) and \(f' = \varphi_{ij}(b)f\) for some indices \(i\) and \(j\). Refer to \cite[Proposition 5.2]{knI} for further details.

It is also important to note that if the fibers \(F\) are homogeneous, it is proved in \cite{Guijarro2007} that by fixing a point \(b_0 \in B\), we can describe the bundle \(P\) as the set of all diffeomorphisms \(h: F_{b_0} \rightarrow F_b\) for \(b \in B\). Here, the bundle projection \(\pi_P(h: G_{b_0} \rightarrow F_b)\) is given by \(\pi_P(h) = b\). This description holds true when \(\pi\) is a fat Riemannian submersion with totally geodesic fibers.    

    From $P$ we can consider another principal $G$-bundle $\overline\pi:P\times F \to M$ whose principal $G$-action is given by
\begin{equation}\label{eq:action}
	r\star(p,f) := (r\cdot p, rf),~r\in G,~p\in P,~f\in F.
	\end{equation} 
	(For the details, see the construction on the proof of \cite[Proposition 2.7.1]{gw}.)
	
Let $\ga_P$ and $\textsl g_F$ be a pair of $G$-invariant metrics on $P$ and $F$, respectively. We can consider the metric $\ga$ on $M$ as the unique (up to scale) connection metric which makes \begin{equation}\label{eq:projection}
\overline\pi : (P\times F,\ga_P+\textsl g_F)\rightarrow (M,\ga)
\end{equation} a Riemannian submersion. Denote by $\mathcal{M}$ the set of all such metrics $\ga$ obtained in the above manner, i.e., $\cal M:=\{\ga=\bar\pi_{\ast}(\ga_P{+}\ga_F)\}$ where $\ga_P$ is a $G$-invariant Riemannian metric on $P$, and $\ga_F$ is a $G$-invariant Riemannian metric on $F$.
\begin{lemma}[Proposition 1.6 in \cite{Ziller_fatnessrevisited}]\label{lem:inducedmetric}
    Let $F\rightarrow (M,\ga)\stackrel{\pi}{\rightarrow} (B,\ga_B)$ be a Riemannian submersion with totally geodesic fibers. If the holonomy group $H$ (at some point) of $\pi$ acts transitively by isometries on $F$, then $\ga$ belongs to $\cal M$. That is, there exists $H$-invariant metrics $\ga_P$ on $P$ and $\ga_F$ on $F$ such that $\ga=\bar\pi_{\ast}(\ga_P+\ga_F)$. 
\end{lemma}

	\begin{definition}[Generalized Cheeger deformations for fat Riemannian submersions]\label{defn}
 Let $F\rightarrow (M,\ga)\stackrel{\pi}{\rightarrow} (B,\ga_B)$ be a fat Riemannian submersion. Let $\textsl g_F$ be the $H$-invariant metric on $F$ and $\ga_P$ be the $H$-invariant metric on $P$ which induces $\ga$ -- Lemma \ref{lem:inducedmetric}.  For each $t\geq 0$ let $(\ga_P)_t$ as the time $t$ Cheeger deformation of $\ga_P$.
 
 The \emph{Cheeger deformation of $\ga$ is the unique (up-to-scale) Riemannian metric} $\ga_t$ that makes $\bar\pi: (P\times F,(\ga_P)_t+\ga_F)\rightarrow (M,\ga_t)$ to be a Riemannian submersion.
	\end{definition}

Fix $(p,f) \in P\times F$ and consider any $\overline{X}\in T_{(p,f)}(P\times F)$. Let $\mathfrak{m}_f$ denote the tangent space of the orbit of $G$ at $f\in F$. We can express $\overline{X}$ as $(X+V^{\vee},X_F+W^*)$, where $X$ is orthogonal to the $G$-orbit on $P$ and $X_F$ is orthogonal to the $G$-orbit on $F$. In this context, $V^{\vee}$ and $W^{\ast}$ represent the action vectors relative to the $G$-actions on $P$ and $F$, respectively. We abuse notation identifying $\mathrm{d}\bar\pi_{(p,f)}(X,X_F+U^*)\equiv X+X_F+ U^*$. 
\begin{definition}\label{def:orbittensors2}
    Let $O$, $O_F$, and $O_t$ be the orbit tensors associated with $\textsl g$, $\textsl g_F$, and $\textsl g_t$ -- Definition \ref{def:orbittensors}. Similarly to the orbit and metric tensors in a Cheeger deformation, we consider those corresponding to the deformation defined in Definition \ref{defn}. Specifically, we define $\widetilde O_t$ and $\widetilde C_t:\mathfrak{m}_f\to \mathfrak{m}_f$ in terms of $O, O_F, O_t$:
	\begin{align*}
	\widetilde O_t & := O_F(1+O_t^{-1}O_F)^{-1}=(O_F^{-1}+O_t^{-1})^{-1},\\
	\widetilde C_t & := -C_tO_t^{-1}\widetilde O_t=-O^{-1}\widetilde O_t.
	\end{align*}	
\end{definition}
  \begin{lemma}[Claim 3 in \cite{Cavenaghi2022}]
     \label{claim:lift}
		Let $\cal L_{\overline\pi}: T_{\overline\pi(p,f)}M \to T_{(p,f)}(P\times F)$ be the horizontal lift associated with $\overline\pi$. Then,
		\begin{equation}
		\cal L_{\overline\pi}(X+X_F+U^*)=(X-(O_t^{-1}\widetilde O_tU)^\vee,X_F+(O_F^{-1}\widetilde O_tU)^*).	
		\end{equation}
 \end{lemma}

 From the introduced concepts, one can infer a useful formula for the sectional curvature of fat Riemannian submersions:
\begin{theorem}[Theorem 3.1 and Lemma 3 in \cite{Cavenaghi2022}]
Let $\pi: F\hookrightarrow (M,\ga)\rightarrow (B,\ga_B)$ be a fat Riemannian submersion. The sectional curvature $K_{\ga}$ of $\ga$ can be expressed as
     \begin{equation}\label{eq:necessaria}
	K_{\ga}(\widetilde X,\widetilde Y)=K_{\ga_P}(X{+}U^{\vee},Y{+}V^{\vee}){+}K_{\textsl g_F}(X_F - (O_F^{-1}OU)^*, Y_F - (O_F^{-1}OV)^*){+}\widetilde z_0(\widetilde X,\widetilde Y).
	\end{equation}
	where $\widetilde z_0$ is a non-negative term. If $X_F=Y_F=V^*=0$ then $\widetilde z_0(\widetilde X,\widetilde Y)=0$.  
\end{theorem}

 \section{Some new rigidity results for fat bundles}
\subsection{Fat Riemannian submersions}
 Our first new structure result is the following.

 \begin{proposition}\label{prop:finishing}
Consider the fat Riemannian submersion $\pi: F\hookrightarrow(M,\ga)\rightarrow (B,\ga_B)$. If $M$ has non-negative sectional curvature ($\mathrm{sec}(\ga) \geq 0$) and $F$ is a symmetric space of dimension greater than one, then $F$ has positive sectional curvature at the normal homogeneous space metric\footnote{Recall that a normal homogeneous metric on a homogeneous space $H/K$ is the quotient metric induced from a bi-invariant metric on $H$ where $K$ acts isometrically, effectively and transitively}.
\end{proposition}
\begin{proof}
    
Once $\pi$ is fat, $F$ can be represented as a homogeneous space $H/K$, where $H$ is the holonomy group of the submersion $\pi$ at some point $b\in B$. Proposition \ref{prop:fatfiber} ensures $M\cong P \times_H (H/K)$, where $H\hookrightarrow P \rightarrow P/H$ is the $\pi$-associated holonomy principal bundle. Picking a bi-invariant metric $Q_H$ on $H$, consider the Riemannian submersion below, where $\bar Q$ stands for the normal homogeneous space metric induced from $Q_H$:
\[
K \hookrightarrow (H,Q_H) \rightarrow (H/K,\bar Q).
\]
Being $Q_H$ a bi-invariant metric, it has non-negative sectional curvature. Thus, $\bar Q$ also has non-negative sectional curvature. Next, we study the curvature of a non-degenerate vertizontal plane tangent to $M$ for the metric $\ga$. 

Once $\ga$ is induced via a choice of $H$-invariant metric $\ga_P$ on $P$ that makes  $\bar \pi: (P\times (H/K),\ga_{P}+\bar Q)\rightarrow (M,\ga)$ be a Riemannian submersion, we rely on Equation \eqref{eq:necessaria}. A non-degenerate vertizontal plane can be written in the form $X\wedge V^*$ for $V\in \lie h\ominus \lie k:= \lie h\cap (\lie k)^{\perp_{Q_H}}$, where $\lie h$ is the Lie algebra of $H$ and $\lie k$ the Lie algebra of $K$. Decomposing $TP = \cal H^P\oplus \cal V^P$ according to the submersion $H\hookrightarrow P\rightarrow P/H$, let $(A^P)^*_X$ be the $\ga_P$-dual to $\h[X,\cdot]^{\cal V^P}$. We have
\[
K_{\ga}(X,V^*) = K_{\ga_P}(X,V^{\vee}) = |(A^P)^*_XV^{\vee}|{\ga_P}^2.
\]
Therefore, $\pi : M\cong P\times_H(H/K) \rightarrow P/H$ is fat if, and only if, the submersion $\pi : P \rightarrow P/H$ is \emph{$K$-fat}. That is, the following two-form is non-degenerate for every non-zero $U \in \lie h\ominus \lie k$:
\begin{equation}
(X, Y) \mapsto Q_H(\Omega(X,Y),U),
\end{equation}
where $(\Omega(X,Y))^{\vee} = -[X,Y]^{\cal V^P}$ (compare with Definition 2.8.2 in \cite[p. 106]{gw}). Hence, for any fixed $0 \neq X\in \cal H^P_p \subset T_pP$, there is an injection $(\Omega(X,\cdot))^{\top} : \cal H^P_p\cap \{X\}^{\perp} \rightarrow \lie h\ominus\lie k$, given by the $Q_H$-orthogonal projection of $\mathrm{Im}~\Omega(X,\cdot)$ onto $\lie h\ominus \lie k$. On the other hand, following Proposition \ref{prop:dimcons}, it holds $\dim \lie h\ominus \lie k = \dim H/K \leq \dim P/H -1 = \dim \cal H^P-1$. Thus, $\dim P/H-1 = \dim H/K$ and $\dim P/H = 2, 4, 8$, and so $\dim H/K = 1, 3, 7$. We finish the proof using this former information and a contradiction argument.

 Assuming that $\dim(F{=}H/K) >1$, we can find vectors $U$ and $V$ in $\lie h\ominus \lie k$ with $|U|_{Q_H} = |V|_{Q_H} = 1$ and $Q_H(U,V) = 0$. As $H/K$ is a symmetric space, we have $[U,V]^{\lie h\ominus \lie k} = 0$. By O'Neill's submersion formula, we get
\[
\mathrm{sec}_{\bar Q}(U,V) = \frac{1}{4}|[U,V]|^2_{Q_H} +\frac{3}{4}|[U,V]^{\lie k}|_{Q_H}^2 = |[U,V]^{\lie k}|_{Q_H}^2.
\]
We claim that $[U,V]^{\lie k}\neq 0$. 

Assuming by contradiction that $[U, V]^{\lie k} = 0$, then $U$ and $V$ lie in the Cartan subalgebra of $\lie h$, and $H/K$ is a symmetric space of rank greater than $1$. If $\dim(P/H) = 4$, then $\dim(H/K) = 2, 3$ so $\mathrm{rank}(H/K) = \dim(H/K) = 2$ implies that either $H/K$ is diffeomorphic to a Euclidean space, which contradicts compactness, or it is not simply connected and hence diffeomorphic to a torus. In the former case, the induced map at the fundamental group level
\[
(\bar\pi)_{\ast} : \pi_1(P\times H/K,(p,f)) \rightarrow \pi_1(M,\bar\pi(p,f))
\]
from $\bar\pi: P\times F \rightarrow M$ is an epimorphism implying that $M$ has an infinite fundamental group. However, according to our hypotheses, since $\ga$ has non-negative sectional curvature, the Bonnet--Myers Theorem implies that $\ga$ is necessarily flat, which contradicts fatness.

We then must have $\dim{H/K} = 3$. However, this is impossible since there is no simply connected 3-dimensional symmetric space of rank 2 -- see Table 2 in \cite[p. 354]{HELGA}. The case $\dim H/K = 7$ follows similarly. \qedhere
% the Lie algebra of the isotropy subgroup of a simply connected Riemannian symmetric space is reductive. In particular, if the symmetric space has rank 2, then its isotropy Lie algebra is either $\mathfrak{su}(2)$, $\mathfrak{so}(3)$, or $\mathfrak{sp}(1)$, having dimension greater than 2. Hence, they cannot be the Lie algebra of the isotropy subgroup of a simply connected Riemannian symmetric space of dimension 3.
%Finally, observe that $\dim H/K = 7$ is also impossible since then $H/K$ has rank greater or equal to 2, thus being diffeomorphic to a torus.
\end{proof}

\subsection{Fat submersions and ``twisted foliations''}
%This subsection focuses on fat submersions disregarded with specific metrics \emph{a priori}. Namely, we consider submersions 
Let $\pi: F\hookrightarrow M\rightarrow B$ be a submersion with compact fibers and a chosen horizontal distribution $\cal H$. The vertical distribution $\cal V$ collects pointwise $\ker~\mathrm{d}\pi$. We assume that for any $x\in M$ and any non-zero $X\in \cal H_x$, it holds for any horizontal extension $\widetilde X$ of $X$ that \[[\widetilde X,\cal H_x]^{\mathbf{v}}=\cal V_x.\]
%with a compact structure group $G$

 \begin{definition}
     For each $x\in M$ we call
\[L^{\#}_x := \{y \in M : \text{there exists a piecewise smooth horizontal curve from $x$ to $y$}\}\]
the \emph{dual-leaf through $x$}. To the collection $\cal F^{\#} = \{L^{\#}_x : x \in M\}$ we refer \emph{the dual foliation} associated with the foliation $\cal F = \{\text{connected components of the fibers } F_x : x \in M\}$.
 \end{definition}

Collecting the connected components of the fibers of a submersion $F\hookrightarrow M\rightarrow B$ yields a foliation. If a Riemannian metric is in place for which two leaves are locally equidistant, we have a Riemannian foliation. 
\begin{definition}[\cite{Angulo-Ardoy2013-pk}]\label{def:twisted}
    A Riemannian foliation is termed \emph{twisted} if it has only one dual leaf.
\end{definition}

\begin{theorem}\label{cor:guijarao}
    Let $\pi: F\hookrightarrow M\rightarrow B$ be a fat submersion. Pick any Riemannian metric $\ga$ on $M$ making the connected components of $\{F_x:x\in M\}$ a Riemannian foliation $\cal F$. Then, $\cal F$ is twisted. 
\end{theorem}
We must digress to prove Theorem \ref{cor:guijarao}, introducing other required results. Before this intercourse, we remark that Theorem \ref{cor:guijarao} generalizes the examples in \cite{Angulo-Ardoy2013-pk}, including
\begin{theorem}[Corollary 3 in \cite{Angulo-Ardoy2013-pk}]
    Every principal fat bundle $M$ over $B$ with an invariant metric of
non-negative curvature is twisted.
\end{theorem}

In principal bundles, once a \( G \)-invariant connection \( \widetilde{\mathcal{H}} \) is fixed with a curvature two-form \( \Omega \), the Ambrose–Singer Theorem ensures that, after proper identification, the intersection \( TL_x^{\#} \cap \mathcal{V} \) is spanned by \( \Omega(X, Y) \), where \( X, Y \in \Gamma(\widetilde{\mathcal{H}}) \). In the case of manifolds with foliations, however, the connection and curvature two-form are generally absent. Instead, a natural alternative is the O’Neill A-tensor, whose image resides in different fibers of \( \mathcal{V} \). Sperança achieved a result (Theorem \ref{thm:llohannunpublished}) that can be considered a variation of the Ambrose--Singer result. See Appendix \ref{ap:llohann} for a proof sketch.
\begin{theorem}[Theorem 1.3 in \cite{speranca_grove}]\label{thm:llohannunpublished}
	Let $(M,\ga)$ be a compact Riemannian manifold with non-negative sectional curvature with a Riemannian foliation $\cal F$ of totally geodesic leaves. For each each $x\in M$ it holds that \[T_xL_x^{\#}\cap \cal V_x = \mathrm{span}\left\{[X,Y]^{\mathbf{v}} : X,Y \in \cal H_x\right\}.\] 
	\end{theorem}

In general, Riemannian foliations need not be globally recovered as the connected components of the fibers of a Riemannian submersion. However, infinitesimal data can be recovered from \emph{holonomy fields} -- Equation \eqref{eq:holonomyfield}. Let $(M,\ga)$ be a Riemannian manifold with a Riemannian foliation $\cal F$. As in the case of Riemannian submersions, we decompose $TM=\cal V\oplus\cal H$ where $\cal V$ comprises the sub-bundle collecting tangent vectors to the leaves of $\cal F$, $\cal H$ is $\ga$-orthogonal to $\cal V$ pointwise. Let $c$ be a horizontal curve, i.e., $\dot c(t)\in\cal H_{c(t)}$. Therein we keep denoting by $A$ the O'Neill tensor $\cal H_x\times \cal H_x\ni (X,Y)\mapsto \h[X,Y]^{\mathbf{v}}=:A_XY,~\forall x \in M$ and by $A^*_X$ the $\ga$-dual to $A_X$. We name a Riemannian foliation $\mathcal F$ \emph{fat} if Equation \eqref{eq:fat} holds analogously.
\begin{definition}[Holonomy fields and Dual Holonomy fields]
    An \emph{holonomy field} $\xi$ along $c$ is a vertical field satisfying
\begin{equation}\label{eq:holonomyfield}
\nabla_{\dot c}\xi=-A^*_{\dot c}\xi-S_{\dot c}\xi,
\end{equation}
where $S\co \Gamma(\cal H)\times\Gamma(\cal H)\to \Gamma(\cal V)$ is the \emph{second fundamental form of the fibers:}
\begin{equation*}
S_X\xi=-(\nabla_\xi \widetilde X)^{\mathbf v}
\end{equation*}
with $\widetilde X$ any horizontal extension of $X$. A vertical field $\nu$ along $c$ is called a \emph{dual holonomy field} if\begin{equation}\label{eq:dualhol}
\nabla_{\dot c}\nu=-A^*_{\dot c}\nu+S_{\dot c}\nu.
\end{equation}
\end{definition}

For a horizontal curve $c:[0,1]\to M$, let $h:\cal V_{c(0)}\to \cal V_{c(1)}$ be the linear isomorphism given by $h(\xi_0)=\xi(1)$, where $\xi(t)$ is the holonomy field along $c$ with initial condition $\xi(0)=\xi_0$. Following \cite{speranca2017on}; we call $h$ an \textit{infinitesimal holonomy transformation}.

Let $\cal E$ collect all infinitesimal holonomy transformations defined by $\cal F$. $\cal E$ is naturally included in 
\begin{gather}
\rm{Aut}(\cal V)=\{h:\cal V_x\to\cal V_y~|~x,y\in M, h\in\rm{Iso}(\cal V_x,\cal V_y)\} ,
\end{gather}
where $\rm{Iso}(\cal V_x,\cal V_y)$ stands for the set of linear isomorphisms between $\cal V_x$ and $\cal V_y$. The natural operations on $\rm{Aut}(\cal V)$ define a groupoid structure, the source and target maps being defined on $h:\cal V_x\to \cal V_y$ by $\sigma(h)=x$ and $\tau(h)=y$, respectively. Moreover, $\cal E$ is closed by  composition and inversion in $\rm{Aut}(\cal V)$: if $h:\cal V_x\to \cal V_y$ is realized by the horizontal curve $c$ and $h':\cal V_y\to \cal V_z$ is realized by $c'$, then $h'\circ h$ is realized by the concatenation of $c$ and $c'$; $h^{-1}$ is realized by the curve $\tilde c$ defined by $\tilde c(t)=c(1-t)$.
The identity section $p\mapsto \rm{id}_{\cal V_x}$ is realized by constant curves.

The topology considered in $\rm{Aut}(\cal V)$ is the topology defined by the submersion $\sigma\times \tau:\rm{Aut}(\cal V)\to M\times M$ along with the operator norm on $\rm{Iso}(\cal V_x,\cal V_y)$ induced by the metric on $M$.  The space $\cal E$  inherits a topology and a groupoid structure from $\rm{Aut}(\cal V)$. 
The analogous corresponding of the ``compact holonomy group'' for general Riemannian foliations is given in the following.
 \begin{definition}[Bounded Holonomy]\label{def:bhol}
We say that a Riemannian foliation has bounded holonomy if there is a constant $L$ such that,  for every holonomy field $\xi$, $|\xi(1)|\leq L|\xi(0)|$.
\end{definition}
    \begin{example}[Proposition 3.4 in \cite{speranca2017on}]
     If the structure group of a given Riemannian submersion $\pi:(M,\ga)\rightarrow (B,\ga_B)$ with compact connected total space is compact, the Riemannian foliation induced by the connected component of the fibers has bounded holonomy. In particular, any fat Riemannian foliation satisfies the bounded holonomy condition.
\end{example}

			\begin{theorem}[Theorem 6.2 in \cite{speranca2017on}]\label{llohannzinho}
	Let $(M,\ga)$ be a compact connected Riemannian manifold with a Riemannian foliation $\cal F$ of bounded holonomy and positive vertizontal curvature. Then, the dual foliation $\cal F^{\#}$ has only one leaf. 
	\end{theorem}

 We are finally in the position to prove Theorem \ref{cor:guijarao}.
\begin{proof}[Proof of Theorem \ref{cor:guijarao}]
    By hypotheses, the Riemannian foliation $\cal F$ given by the connected components of the fibers $\{F_x\}_{x\in M}$ is fat, so that Proposition \ref{prop:fatfiber} ensures that $F$ is diffeomorphic to the homogeneous space $F = H/K$ where $H$ is the holonomy group of $\pi$. It also holds $M \cong P\times_H(H/K)$ where $P$ is the $\pi$-associated reduced holonomy bundle.

Denote by $\bar \pi : P\times (H/K)\rightarrow M$ the quotient projection. Since $F$ is a homogeneous manifold for which $H$ acts isometrically, it holds that there exists $\textsl g_P$ and $\textsl g_F$ such that $\ga$ decouples as a product by the $\bar \pi$-pullback:
\[{\bar{\pi}}^{\ast}(\ga) = \textsl{g}_P+\textsl g_F,\]
where $\textsl g_P$ is a $H$-invariant metric on $P$ and $\textsl g_F$ is an $\mathrm{Ad}(K)$-invariant metric on $F$ -- this is the content of Lemma \ref{lem:inducedmetric}. Performing corresponding Cheeger deformations on $\ga_P$ and $\ga_F$ followed by a canonical variation of each metric will induce, via the deformation introduced in Definition \ref{defn}, a one-parameter family of Riemannian metrics on $M$, which converges to a metric with totally geodesic leaves -- \cite[Theorem 3.2]{Cavenaghi2022}. Theorem \ref{llohannzinho} ensures the existence of only one dual leaf to the corresponding dual foliation since the horizontal distribution is never changed for the described procedure. \qedhere
\end{proof}

 In the Appendix \ref{ap:llohann}, for readers' convenience and completeness, we both sketch a proof for Theorem \ref{thm:llohannunpublished} and add some comments on Riemannian foliations on non-negatively curved manifolds. Before moving to the next section, as shall be useful, we observe that the combination of the Theorems \ref{thm:llohannunpublished} and \ref{llohannzinho} recovers a geometrical connection between Ambrose--Singer's Theorem and fatness:	
	\begin{corollary}\label{cor:symmetricimposed}
	Let $(M,\ga)$ be a compact connected Riemannian manifold with a fat Riemannian foliation $\cal F$ of bounded holonomy and totally geodesic leaves. If $\ga$ has non-negative sectional curvature then for each $x\in M$ it holds that $\cal V_x = \mathrm{span}\left\{A_XY(x) : X,Y \in \cal H_x\right\}$.
	\end{corollary}

\subsection{A rigidity result for non-negatively curved fat Riemannian submersions}

As our last result for this section, we prove
\begin{theorem}\label{thm:symmetricfibers}
The fiber $F$ of a fat Riemannian submersion with totally geodesic leaves $\pi: F\hookrightarrow (M,\ga) \rightarrow (B,\ga_B)$ is a symmetric space if $\mathrm{sec}(\ga) \geq 0$.
\end{theorem}
\begin{proof}
Since the fibers are totally geodesic and $\ga$ has non-negative sectional curvature, it must hold that for every $x \in M$ (Corollary \ref{cor:symmetricimposed}): \begin{equation}\label{eq:aquiobstruiu}\cal V_x =: T_x(K/H) \cong \mathrm{span}_{\bb R}\left\{A_XY(x): X, Y \in \cal H_x\right\}.\end{equation}

Using that $F$ is a homogeneous space obtained out of the holonomy group of the foliation, it must hold that $\ga$ is obtained out of a product metric $\textsl{g}_P{+}\textsl{g}_F$ on $P\times F$ such that $\overline\pi:(P\times F,\textsl g_P + \textsl g_F)\to (M,\textsl g)$ is a Riemannian submersion. Moreover, any $\overline{X}\in T_{(p,f)}(P\times F)$ can be written as $\overline{X} = (X+V^{\vee},X_F+W^*)$, where $X$ is orthogonal to the $H$-orbit on $P$ (that is, it is horizontal for the metric $\ga_P$), $X_F$ is orthogonal to the $H$-orbit on $F$ (that is, horizontal for the metric $\ga_F$) and, for $V,W\in\lie h$, the vectors $V^{\vee}$ and $W^{\ast}$ are the action vectors relative to the $H$-actions on $P$ and $F$, respectively. Note that a horizontal vector in $T_xM$ can be written as $X+X_F$.

Let $\cal L_{\overline\pi}: T_{\overline\pi(p,f)}M \to T_{(p,f)}(P\times F)$ be the horizontal lift associated with $\overline\pi$, with $\bar \pi(p,f) = x$. If $\mathrm{p}_1 : TP\times TF\rightarrow TP$ denotes the first factor projection, Lemma \ref{claim:lift} ensures that $\mathrm{p}_1\circ \cal L_{\bar \pi}(X+X_F) = X$. Therefore, $\mathrm{p}_1\circ \cal L_{\bar \pi}$ maps the horizontal space $\cal H_x$ onto the horizontal space of $P$ at $p$ with respect to $\ga_P$. Since $F$ is a homogeneous space, it holds that $X_F \equiv 0$ and hence, $\mathrm{p}_1\circ \cal L_{\bar \pi}$ defines an isomorphism between the horizontal space $\cal H_x$ and the $\ga_P$-orthogonal complement to the $H$-orbit through $p$, that we denote by $\cal H^P_p$.

Fixed $h\in H$, denote by $\rho_h : P \rightarrow P$ the $h$-action on $P$. That is, $\rho_h(p) = h\cdot p$. Observe that $(\rho_h)_{\ast}(\cal H^{P}_p) \cong \cal H^{P}_{h\cdot p}$. Consider the induced isometric immersion metric on $F$ and denote by $T_{h\cdot f}F$ the tangent space of $F$ at $h\cdot f$.  

If $X, Y \in \cal H^P_p$ then $d\rho_h(X), d\rho_h(Y) \in \cal H^P_{h\cdot p}$ and using the isomorphism between $\cal H_x$ and $\cal H^P_{h\cdot p}$ we abuse notation to see that
$\cal V_x\cap T_{h\cdot f}F$  is spanned by $A_{d\rho_hX}d\rho_hY$. Moreover, since $H$-defines an isometric action on the fiber $F$, keeping in mind the notation abuse we can check that $A_{d\rho_hX}d\rho_hY = (\rho_h)^{\ast}(A_XY)$. That is,
\begin{align*}
\cal V_x\cap T_{h\cdot f}F\cong \mathrm{span}_{\bb R}\left\{d\rho_{h}\left(A_{X}Y\right)(x) : X, Y \in \cal H_x\right\}.
\end{align*}

Let $\phi : \cal V_x\cap T_fF \rightarrow \cal V_{h\cdot x}\cap T_{h\cdot f}F$ be the linear isometry obtained by extending linearly the map
\[\phi(A_XY(x)) := -(\rho_h)^{\ast}(A_XY(x)), X, Y \in \cal H_x.\]
To see that this is an isometry, it suffices to note that
\[|\phi(A_XY(x))|_{\ga_F} = |(\rho_h)^{\ast}(A_XY(x))|_{\ga_F} = |A_XY(x)|_{\ga_F},\]
since the $H$-action on the fibers is isometric.

Let $\sigma(f') := \exp^F_{h\cdot f}~\circ~\phi~\circ~(\exp_f^F)^{-1}(f')$ where $\exp^F_f$ denotes the exponential map of the isometrically induced metric on $F$ with domain at $T_fF$ and $\exp^F_{h\cdot f}$ is the corresponding exponential map with domain at $T_{h\cdot f}F$. We choose the domains on the composition so that both $\exp_f^F, \exp_{h\cdot f}^F$ are diffeomorphisms. Fixed $f'$ 
sufficiently close to $f$, assume that $\gamma : (-\epsilon,\epsilon) \rightarrow F$ is the minimizing geodesic between $f$ and $f'$.

Let $P_{\gamma}$ be the parallel transport along it, and $P_{\widetilde \gamma}$ be the parallel transport along $\sigma\circ \gamma$. We claim that 
\[I_{\gamma,\widetilde \gamma} := P_{\widetilde \gamma}~\circ~\phi~\circ~P_{\gamma}^{-1}\]
commutes with the Riemann tensor of $\ga_F$. This implies that $\sigma$ is not only a local isometry but that $d\sigma(f') = I_{\gamma,\widetilde\gamma}$ for any $f'$ in the domain of $\sigma$. Once, by construction, $\sigma$ reverses geodesics, the claim shall follow.

Once $A$ is a tensor, let us extend $X,Y \in \cal H_x$ basically along $\gamma$ requiring that if $\overline X, \overline Y$ are such extensions, then $\nabla^{\cal H_{\gamma(s)}}_{\overline X(s)}\overline Y(s) \equiv 0$. Since both $\overline Y(s), A^*_{\overline X}\dot{\gamma}(s)$ are basic along $\gamma$, the inner product
\[\ga(A_{\overline X}{\overline Y(s)},\dot{\gamma}(s)) = \ga(\overline Y(s),A^*_{\overline X}\dot{\gamma}(s))\]
is constant along $\gamma$, so $\frac{\nabla}{ds}A_{\overline X(s)}{\overline Y(s)}$ is always orthogonal to $\dot{\gamma}(s)$. 

The fatness assumption implies that $A_{\overline X(s)}{\overline Y(s)}$ is a Jacobi field along $\gamma$. Since $H$ has non-negative sectional curvature, then Wilking's Theorem \cite[Theorem 1.7.1]{gw} (see also \cite{wilkilng-dual}) ensures that either $A_{\overline X(s)}{\overline Y(s)}$ vanishes at some instant $s$, in which case it vanishes for every $s$, or $A_{\overline X(s)}\overline Y(s)$ is parallel along $\gamma$. If $A_XY(x)$ is non-zero, then $A_{\overline X(s)}{\overline Y(s)}$ is the parallel transport of $A_XY$ along $\gamma$, so $A_{\overline X(s)}\overline Y(s)$ is a Jacobi field with constant coefficients chosen from any parallel orthonormal frame along $\gamma$. The same argument can be applied to $P_{\widetilde \gamma}$. Hence, $I_{\gamma,\widetilde \gamma}$ is a linear isometry.
\end{proof}

 \section{On the fibers of non-negatively curved fat submersions}
In this final section, we achieve our paper's final goal: To classify the possible fibers on fat Riemannian submersion from non-negatively curved manifolds and classify fat Riemannian submersion from Lie groups and locally symmetric spaces.

	\subsection{Riemannian Foliations on Lie Groups with bi-invariant metrics}
\label{sec:liegroups}
In 1986, Ranjan asked if a Riemannian submersion with totally geodesic leaves $\pi:(G, Q)\rightarrow(B,\ga_B)$ from a compact \emph{simple} Lie group with a bi-invariant metric $Q$ is a \emph{coset foliation}, i.e., if the cosets give it for the action of a certain subgroup $H<G$. As Riemannian submersions are the primary examples of Riemannian foliations, Grove posed the problem of classifying Riemannian submersions from compact Lie groups with bi-invariant metrics (Problem 5.1 in \cite{grove_problems}).

Most examples of manifolds with positive sectional curvature are related to Riemannian submersions from Lie groups. Moreover, for a compact Lie group $G$ with a bi-invariant metric, it is known that homogeneous foliations are Riemannian and have totally geodesic leaves. Therefore, it is natural to inquire whether every Riemannian foliation with totally geodesic leaves is homogeneous.
In \cite{speranca_grove}, Sperança answered Ranjan's question in a more general setting, assuming such a submersion is defined only in an open subset of $G$ (with further compactness hypotheses). 

	\begin{theorem}[Theorem 1.1 in \cite{speranca_grove}]\label{thm:grove}
	Any totally geodesic Riemannian foliation on a compact connected Lie group $G$ with a bi-invariant metric $Q$ is isometric to the coset foliation induced by a subgroup $H$ of $G$.
	\end{theorem}

  Related to Theorem \ref{thm:grove} and Grove's problem, we prove a classification of fat foliations on Lie groups induced by the connected component of the submersion's fibers. The non-metric nature of the fatness condition combined with Theorem \ref{cor:guijarao} settles our goal.

 \begin{theorem}\label{thm:zillerq}
	Let $G$ be a compact connected Lie group. Then, any foliation $\mathcal{F}$ induced by the connected components of the fibers of a fat submersion $\pi:G\rightarrow B$ is induced by a Lie subgroup $H<G$. If $\dim H>1$ then $H$ is either $\mathrm{SO}(3)$ or $\mathrm{S}^3$.
	\end{theorem}
 \begin{proof}
Let $G$ be a compact connected Lie group and assume a foliation $\mathcal{F}$ on $G$ as in the hypothesis. Theorem \ref{cor:guijarao} ensures that $\cal F$ is twisted for any Riemannian metric. Theorem 6.5 in \cite{speranca2017on} yields the existence of a bi-invariant metric $Q$ in $G$ for which the leaves of $\cal F$ are totally geodesic. Theorem \ref{thm:grove} ensures that $\cal F$ is isometric to the coset foliation $\{Hg\}_{g\in G}$ induced by the holonomy group $H$ of $\pi$, concluding the first part of the statement.

Finally, Proposition \ref{prop:finishing} ensures that if $\dim F>1$, then $H$ is a positively curved Lie group in the bi-invariant metric, being the only possibilities $\mathrm{SO}(3),~\mathrm{S}^3$.\qedhere
	\end{proof}

 \subsection{On the fibers of non-negatively curved fat Riemannian submersions}
 We prove the following:
 \begin{theorem}\label{ithm:symmetricfiber}
       Let $\pi:(M,\ga)\rightarrow (B,\ga_B)$ be a fat Riemannian submersion with fibers $F$ from a compact connected manifold. Then the holonomy group $H$ of $\pi$ is one of the following, with possible fibers $H/K$:
\begin{table}[h!]
      \begin{center}
          \begin{tabular}{|c|c|c|c|}
\hline $H/K$ & $H$        &$K$ & $\dim H/K$\\
\hline \hline
\hline
$\mathrm{SO}(3)$ & $\mathrm{SO}(3)$ & $\{e\}$  & $3$\\
\hline
$\mathrm{S}^2$ & $\mathrm{SO}(3)$ & $\mathrm{SO}(2)$ & $2$\\
\hline
$\mathrm{S}^2$ & $\mathrm{SU}(2)$ & $\mathrm{SO}(2)$ & $2$\\
\hline
$\mathrm{S}^2$ & $\mathrm{Sp}(1)$ & $\mathrm{U}(1)$ & $2$\\
\hline
$\mathrm{S}^3$  & $\mathrm{Sp}(1)$ & $\{e\}$ & $3$\\
\hline
$\mathrm{S}^3$ & $\mathrm{Spin}(3)$ & $\{e\}$ & $3$\\
\hline
$\mathrm{S}^3$ & $\mathrm{SU}(2)$ & $\{e\}$ & $3$ \\
\hline
$\mathrm{S}^5$ & $\mathrm{SU}(4)$ &$\mathrm{Sp}(2)$ & $5$\\
\hline
$\mathrm{S}^{n(n-1)}$ & $\mathrm{SO}(2n)$ & $\mathrm{U}(n)$ & $n(n-1)$ with $n=2,3$\\
\hline
$\bb{RP}^3$ & $\mathrm{SU}(2)$ & $\bb Z_2$  & $3$\\
\hline
$\bb{CP}^n$ & $\mathrm{SU}(n+1)$ & $\mathrm{S}(\mathrm{U}(n)\times \mathrm{U}(1))$ & $2n$\\
\hline
$\bb{HP}^n$ & $\mathrm{Sp}(n+1)/\bb Z_2$ & $\left(\mathrm{Sp}(n)\times \mathrm{Sp}(1)\right)/\bb Z_2$ & $4n$\\
\hline
$\mathrm{Ca}\bb P^2$ & $\mathrm{F}_4$ &$\mathrm{Spin}(9)$ & $16$\\
\hline
$\bb{RP}^{2n}$ & $\mathrm{SO}(2n+1)$ & $\mathrm{SO}(2n)$ & $2n$\\
\hline
         \end{tabular} 
         \end{center}
\caption{Fibers constraint}
\label{table:fiber}
\end{table}
\FloatBarrier
 \end{theorem}
 \begin{proof}
     Theorem \ref{thm:symmetricfibers} and Proposition \ref{prop:finishing} combined ensure that if $\dim{H/K} > 1$, then $H/K$ is a symmetric space of compact type with positive sectional curvature. Theorem 4.5 in \cite{CHAVEL1970236}, on its turn, guarantees that $H/K$ is necessarily a rank one symmetric space that is two-point homogeneous -- a homogenous space $K/H$ is said \emph{two-point homogeneous} if for any $x, y, z, w \in K/H$ satisfying $\mathrm{dist}(x,y) = \mathrm{dist}(z,w)$, where $\mathrm{dist}$ is the metric induced distance function, there is $x^* \in G$ such that $x^*\cdot x = z$ and $x^*\cdot y = w$. The $\cdot$ operation denotes the $K$-action on $K/H$. Theorem 4.6 in \cite{CHAVEL1970236} asserts that $H/K$ has constant sectional curvature or is $1/4$-pinched. Lemma 1.2 in \cite{CHAVEL1970236} ensures that the rank of $H$ and $K$ is the same. 

     According to Theorem 4.8 in \cite{CHAVEL1970236}, if $H/K$ is simply connected and odd-dimensional, it is a sphere of constant sectional curvature. If $H/K$ is not simply connected (the dimension has no parity assumption), then $H/K$ is the real projective space of constant curvature. Theorem 4.9 in \cite{CHAVEL1970236} also guarantees that if $H/K$ is not isometric to a sphere of constant sectional curvature, its dimension is a multiple of $2, 4, 8$. These pieces of information gathered impose the desired restriction. \qedhere
 \end{proof}

 Requiring that $M$ is a locally symmetric space, one immediately gets from the former:
 \begin{corollary}
    \label{ithmq:ziller}
Let $G\hookrightarrow (M,\ga)\rightarrow (B,\ga_B)$ be a fat Riemannian principal bundle. If $(M,\ga)$ is isometric to a irreducible rank one symmetric space $H/K$ and $\dim G = 3$ then $B$ is isometric to a symmetric space, so $B = H'/K'$ and $(M, H, K, B, H', K')$ correspond to one of the following 
\begin{table}[h!]
\label{table:ziller}
\begin{center}
    \begin{tabular}{|c|c|c|c|c|c|}
    \hline
       $M$ & $H$ & $K$ & $B$ & $H'$ & $K'$  \\
        \hline
        \hline
$\mathbb{RP}^7$     & $\mathrm{SO}(8)$ & $\mathrm{SO}(7)$ &   $\bb{RP}^4$ & $\mathrm{SO}(5)$ & $\mathrm{SO(4)}$\\
       \hline
 $\mathbb{RP}^{11}$  & $\mathrm{SO}(12)$ & $\mathrm{SO}(11)$ &    $\bb{RP}^8$ & $\mathrm{SO}(9)$ & $\mathrm{SO}(8)$\\ \hline
  $\mathrm{S}^7$ & $\mathrm{SU}(3)$ & $\mathrm{SU}(1)$ &   $\bb{CP}^2$ & $\mathrm{SU}(3)$  & $\mathrm{S}(\mathrm{U}(2)\times \mathrm{U}(1))$\\
      \hline
   $\mathrm{S}^7$ & $\mathrm{Sp}(2)$ & $\mathrm{Sp}(1)$ &  $\mathbb{HP}^1$ & $\mathrm{Sp}(2)$  & $\mathrm{Sp}(1)\times \mathrm{Sp}(1)$ \\\hline
$\mathrm{S}^{11}$ & $\mathrm{Sp}(3)$ & $\mathrm{Sp}(2)$ &    $\mathbb{HP}^2$ & $\mathrm{Sp}(3)$ & $\mathrm{Sp}(2)\times \mathrm{Sp}(1)$\\
     \hline
    \end{tabular}
\end{center}
\caption{Locally symmetric spaces group realization}
    \end{table}
    \FloatBarrier
\end{corollary}

\subsection*{Concluding remarks}
\begin{itemize}
    \item 	The here employed techniques can be used to recover Bérard-Bergey result (\cite{bergery1976varietes}), i.e., if $\pi$ is a fat homogeneous Riemannian submersion \[H/K \hookrightarrow (G/K,\textsl{b}) \rightarrow (G/H,\bar{\textsl b}),\]
 where  $\textsl b$ has nonnegative sectional curvature and  $\dim{H/K} > 1$ then the triple $(K,H,G)$ can be one of the following:
\begin{table}[h!] \begin{center} \begin{tabular}{|c|c|c|}
\hline $G$ &$H$ & $K$ \\
\hline
$\mathrm{ SU }(3)$& $\mathrm{ SO }(3)$ & $\{e\}, \mathrm{ SO }(2)$ \\
\hline
$\mathrm{ SU }(n+2)$& $\mathrm{ SU }(n+1)$ & $\mathrm{S}(\mathrm{U}(n)\times \mathrm{U}(1))$ \\
\hline
$\mathrm{ SU }(2(n+1))$& $\mathrm{Sp}(n+1)$ & $\mathrm{Sp}(n)\times \mathrm{Sp}(1)$ \\
\hline
$\mathrm{ SU }(2n+4)$& $\mathrm{ SO }(2n+4)$ & $\mathrm{ SO }(2n+1)\times \mathrm{ SO }(3)$ \\
\hline
$\mathrm{G}_2$ & $\mathrm{ SO }(4)$ & $\mathrm{U}(2)$\\
\hline
$\mathrm{E}_6$ & $\mathrm{F}_4$ & $\mathrm{Spin}(9)$ \\
\hline
$\mathrm{E}_6$ & $\mathrm{Sp}(4)/\bb Z_2$ & $\left(\mathrm{Sp}(3)\times \mathrm{Sp}(1)\right)/\bb Z_2$ \\
\hline \end{tabular} \end{center}\FloatBarrier
\caption{Constraints to $G, H,$ and $K$}
\end{table} 
\FloatBarrier

 The possibilities to $H, K$ are ensured by Theorem \ref{ithm:symmetricfiber}, described in Table \ref{table:fiber}. On the other hand, we know that $G/H$ is a symmetric space of even dimension (see, for instance \cite[Proposition 3.2, p.17]{Ziller_fatnessrevisited}). If $2\leq \dim H/K \leq 3$, then $\dim G/H$ is a multiple of $4$, while $\dim H/K \geq 4$ implies that $\dim G/H$ is a multiple of $8$ -- Proposition \ref{prop:dimcons}. 
 
  Being $G/H$ a symmetric space of compact type with $H$ as in Table \ref{table:fiber}, one can infer the possible candidates to $G$ out of \cite[Table V, p.518]{HELGA}.
\item As a last observation, one notices from the former tables a relation with Theorem 4.1 in \cite{BOCHENSKI2016132}. The condition of fatness does not behave well under structure group reduction nor extension, except for very specific cases, some of which are described in Table \ref{table:fiber}. This is more evidence of how restrictive the fatness condition is. Our presented results are further evidence with the assumption of non-negatively curved metrics.
\end{itemize}

\section*{Acknowledgments}
The authors thank the anonymous referee for thoroughly evaluating our manuscript; the detailed feedback significantly enhanced its quality.

The S\~ao Paulo Research Foundation FAPESP supports L. F. C grant 2022/09603-9. L. G is partially supported by S\~ao Paulo Research Foundation FAPESP grants 2018/13481-0, 2021/04065-6 and 2023/13131-8. L. F. C would also like to thank the University of Fribourg and Anand Dessai for their hospitality during the initial stages of this work, supported by grant SNSF-Project 200020E\_193062 and the DFG-Priority Program SPP 2026.

\appendix
\section{A sketch of the proof of Theorem \ref{thm:llohannunpublished}}

\label{ap:llohann}

\begin{lemma}\label{lem:WNN}
	Let $\cal F$ be a totally geodesic Riemannian foliation with bounded holonomy on a compact non-negatively curved Riemannian manifold $(M,\ga)$. Then, for every $x\in M$, there is a neighborhood of $x$ and a $\tau>0$ such that
\begin{gather}\label{eq:nn}
\tau |X||Z||A^*_X\xi|\geq |\lr{(\nabla_XA^*\xi )X,Z}|
\end{gather}
for all horizontal vectors $X,Z$ and vertical vector $\xi$.\end{lemma}
\begin{proof}
Given $X,Z\in\cal H$ and $\xi\in\cal V$, O'Neill's equations (\cite[p. 44]{gw}) ensure that the unreduced sectional curvature $K(X,\xi+tZ)=R(X,\xi+tZ,\xi+tZ,X)$ is given by
\begin{equation}\label{eq:quadratictrick}
K(X,\xi+tZ)=t^2K(X,Z)+2t\lr{(\nabla_XA)_XZ,\xi}+|A^*_X\xi|^2.
\end{equation}
Once $\ga$ has non-negative sectional curvature, it follows that $K(X,\xi+tZ)\geq 0$. Therefore, the discriminant of \eqref{eq:quadratictrick} (seen as a polynomial on $t$) must be non-negative. That is
\begin{equation*}
0\leq K(X,Z)|A_X^*\xi|^2-\lr{(\nabla_XA)_XZ,\xi}^2.
\end{equation*}
On small neighborhoods, continuity of $K$ guarantees some $\tau>0$ such that $K(X,Z)\leq \tau|X|^2|Z|^2$. Extending $\xi$ to a holonomy field to the computations, one concludes that
\[\lr{(\nabla_X A)_ZY,\xi}=-\lr{(\nabla_X A^*\xi)Z,Y}\] for all horizontal $X,Y,Z$. \qedhere
\end{proof}

\begin{proposition}\label{prop:A-flatgeo}
Let $\cal F$ be a totally geodesic Riemannian foliation with bounded holonomy on a compact non-negatively curved Riemannian manifold $(M,\ga)$. Let  $X_0\in\cal H_x$ and $\xi_0\in\cal V_x$ be such that $A^*_X{\xi_0}=0$. Then, the holonomy field along $\xi(t)$ defined along $c(t)=\exp(tX_0)$ is such that $A_{\dot c(t)}^*{\xi(t)}=0$ for all $t$.
\end{proposition}
\begin{proof}
Taking $|X_0|=1$ and $Z=A^*_{\dot c}{\xi}$ in \eqref{eq:nn}, we get:
\begin{align}\label{proof:flatgeo1}
\tau|A^*_{\dot c}\xi|^2&\geq \lr{\nabla_{\dot c}A^*_{\dot c}\xi,A^*_{\dot c}\xi}=\frac{1}{2}\frac{d}{dt}|A^*_{\dot c}\xi|^2.
\end{align}
The Gronwall's inequality for $u(t)=|A^*_{\dot c}\xi|^2$ and implies that
\begin{equation*}\label{proof:flatgeo}
|A^*_{\dot c}\xi|^2\leq |A^*_{\dot c(0)}{\xi(0)}|^2e^{2\tau t}
\end{equation*}
for all $t>0$. In particular, if $A^*_{X(0)}{\xi(0)}=0$, then $A^*_{\dot c(t)}{\xi(t)}=0$ for all $t>0$. The same argument works for $t<0$ by replacing $X_0$ with $-X_0$. \qedhere
\end{proof}

 Fixed a holonomy field $\xi(t)$, the proof of Theorem \ref{thm:llohannunpublished} is finished with an understanding of the distribution $\cal D(t)=\ker (A^*\xi\co\cal H_{c(t)}\to \cal H_{c(t)})$. With the help of the previous results, it is possible to prove that:

\begin{proposition}[Proposition 2.6 in \cite{speranca_grove}]\label{prop:Dconstantrank} Let $X_0,\xi_0$ satisfy $A^*_{X_0}{\xi_0}=0$. If $\lambda^2$ is a continuous eigenvalue of $\cal D$ along $c(t)=\exp(tX_0)$, then either $\lambda$ vanishes identically, or $\lambda$ never vanishes. 
\end{proposition}

Theorem \ref{thm:llohannunpublished} follows from the classical Ambrose--Singer's Theorem combined with Proposition \ref{prop:Dconstantrank}. Combining the dual leaf terminology with the results in \cite{speranca2017on}, it is not hard to see that the classical Ambrose--Singer's Theorem can be stated as
\begin{theorem}[Ambrose--Singer]\label{thm:AS}
Let $\cal F$ be a Riemannian foliation on a path connected smooth manifold $M$. Denote $\lie a_y=\mathrm{span}_{\bb R}\{A_XY~|~X,Y\in\cal H_y\}$. Then, for every  $x\in M$,
\begin{equation*}
T L_{x}^\#=\cal H_x\oplus \mathrm{span}_{\bb R}\{\hat c(1)^{-1}(\lie a_{c(1)})~|~c : [0,1]\rightarrow M \text{ horizontal} \},
\end{equation*}
where $\hat c(1)^{-1}$ is the holonomy transport along $c$ from the time $t =1$ to the time $t = 0$.
\end{theorem}
\begin{proof}[Sketch of the proof of Theorem \ref{thm:llohannunpublished}]
Let $x\in M$. Observe that 
\begin{equation*}
\lie a_p^\bot=\{\xi\in\cal V_x~|~A^*_X\xi=0~\forall X\in\cal H_x\}.
\end{equation*}
Note that if $c$ is horizontal curve, then $\hat c(1)(\lie a_p^\bot)=\lie a_{c(1)}^\bot$. Indeed, verifying this for horizontal geodesics suffices since piecewise horizontal geodesics can smoothly approximate $c$. Let $c$ be a horizontal geodesic, $c(0)=x$, and $\xi(t)$ be a holonomy field with $\xi(0)\in\lie a_x^\bot$. Then $\ker A^*{\xi(0)} =\cal H_x$ and $\dim\ker A^*{\xi(t)}$ is constant with respect to $t$ (Proposition \ref{prop:Dconstantrank}). Thus $A^*{\xi(t)}=\cal H_{c(t)}$ for all $t$. This concludes the proof since then $\hat c(1)(\lie a_x)=\lie a_{c(1)}$ thus $\hat c(1)$ is an isometry. \qedhere 
\end{proof}


\begin{thebibliography}{main}

\bibitem{Angulo-Ardoy2013-pk}
P. Angulo-Ardoy, L. Guijarro, and G. Walschap.
\newblock Twisted submersions in non-negative sectional curvature.
\newblock {\em Arch. Math.}, 101(2):171--180, Aug 2013.

\bibitem{bergeryfat}
L. B\'erard-Bergery.
\newblock Sur certaines fibrations d'espaces homog\`enes Riemanniens.
\newblock {\em Compos. Math.}, 30(1):43--61, 1975.

\bibitem{bergery1976varietes}
L. B\'erard-Bergery.
\newblock Les vari{\'e}t{\'e}s {Riemanniennes} homogenes simplement connexes de dimension impairea courbure strictement positive.
\newblock {\em J. Math. Pures Appl.}, 55:47--68, 1976.

\bibitem{BOCHENSKI2016132}
M. Bocheński, A. Szczepkowska, A. Tralle, and A. Woike.
\newblock On fibers of fat associated bundles.
\newblock {\em Colloq. Math.}, 150:131--141, 2017.
 

\bibitem{BOCHENSKI2016131}
M. Bocheński, A. Szczepkowska, A. Tralle, and A. Woike.
\newblock On a classification of fat bundles over compact homogeneous spaces.
\newblock {\em Differ. Geom. Appl.}, 49:131--141, 2016.

\bibitem{E1929}
E.~Cartan.
\newblock Groupes simples clos et ouverts et géométrie Riemannienne.
\newblock {\em J. Math. Pures Appl.}, 8:1--34, 1929.

\bibitem{Cavenaghi2022}
L.~F. Cavenaghi, L. Grama, and L.~D. Speran{\c{c}}a.
\newblock The concept of Cheeger deformations on fiber bundles with compact structure group.
\newblock {\em S\~ao Paulo J. Math. Sci.}, Nov 2022.

\bibitem{cavenaghilinollohann2}
L.~F. Cavenaghi, L. Grama, and L.~D. Sperança.
\newblock The Petersen--Wilhelm conjecture on principal bundles.
\newblock {\em arXiv:2207.10749}, 2022.

\bibitem{CHAVEL1970236}
I. Chavel.
\newblock On Riemannian symmetric spaces of rank one.
\newblock {\em Adv. Math.}, 4(3):236--263, 1970.

\bibitem{Chow1940}
W-L. Chow.
\newblock {\"U}ber systeme von liearren partiellen differentialgleichungen erster ordnung.
\newblock {\em Math. Ann.}, 117(1):98--105, Dec 1940.

\bibitem{finotwostep}
S. Console, A. Fino, and  E. Samiou
\newblock{The Moduli Space of Six-Dimensional Two-Step Nilpotent Lie Algebras}
\newblock {\em Ann Glob Anal Geom}, 27, 17–32 (2005). 

\bibitem{do1992riemannian}
M.P. do~Carmo.
\newblock {\em Riemannian Geometry}.
\newblock Math. Birkh{\"a}user, 1992.

\bibitem{DV2022}
M. Dom{\'i}nguez-V{\'a}zquez, D. Gonz{\'a}lez-{\'A}lvaro, and L. Mouill{\'e}.
\newblock Infinite families of manifolds of positive $k\mathrm{th}$-intermediate {Ricci} curvature with $k$ small.
\newblock {\em Math. Ann.}, Aug 2022.

\bibitem{ivanov1}
M. Fernández, S. Ivanov, and Vicente Mu\~noz.
\newblock{Formality of 7-dimensional 3-Sasakian manifolds}.
\newblock {\em  Annali della Scuola Normale Superiore di Pisa. Classe di scienze},SSN 0391-173X, Vol. 19, N. 1, 2019.

\bibitem{finogueo}
A. Fino, G. Grantcharov, and L. Vezzoni.
\newblock {Astheno–K\"ahler and Balanced Structures on Fibrations}.
\newblock{\em International Mathematics Research Notices}, Volume 2019, Issue 22, November 2019, Pages 7093–7117.

\bibitem{gorodski}
C. Gorodski.
\newblock An introduction to {Riemannian Symmetric Spaces}, 2021.
\newblock Available at https://www.ime.usp.br/~gorodski/ps/symmetric-spaces.pdf.

\bibitem{gray1967pseudo}
A.~Gray.
\newblock Pseudo-Riemannian almost product manifolds and submersions.
\newblock 1967.

\bibitem{grove_problems}
K.~Grove.
\newblock Geometry of and via symmetries.
\newblock {\em Conformal, Riemannian and Lagrangian Geometry: The 2000 Barrett Lectures}, 27. 2011.

\bibitem{Grove_anexotic}
K.~Grove, L.~Verdiani, and W.~Ziller.
\newblock An exotic {$T_1S^4$} with positive curvature.
\newblock {\em Geom. Funct. Anal.}, pages 809--2304. 2011.

\bibitem{Guijarro2007}
L. Guijarro and G. Walschap.
\newblock When is a Riemannian submersion homogeneous?
\newblock {\em Geom. Dedicata}, 125(1):47--52, Mar 2007.

\bibitem{gw}
D.~Gromoll and G.~Walshap.
\newblock {\em Metric Foliations and Curvature}.
\newblock Birkhäuser Verlag, Basel, 2009.

\bibitem{HELGA}
S. Helgason.
\newblock {\em Differential Geometry and {Symmetric Spaces}}.
\newblock 1962.

\bibitem{knI}
S.~Kobayashi and K.~Nomizu.
\newblock {\em Found. Differ. Geom.}, volume~I.
\newblock Interscience Publishers, 1963.

\bibitem{Muter}
M. M\"uter.
\newblock {\em Krumm\"ungserh\"ohende deformationen mittels gruppenaktionen}.
\newblock PhD thesis, Westfälischen Wilhelms-Universität Münster, 1987.

\bibitem{oneill}
B. O'Neill.
\newblock The fundamental equations of a submersion.
\newblock {\em Michigan Math. J.}, 13 (1966), 4:459--469, 1966.

\bibitem{ovando1}
G.~P. Ovando and M. Subils.
\newblock Magnetic fields on non-singular 2-step nilpotent Lie groups.
\newblock {\em arXiv:2210.12180}, 2022.

\bibitem{ovando2}
G.~P. Ovando and M. Subils.
\newblock Symplectic structures on low dimensional 2-step nilmanifolds.
\newblock {\em arXiv:2211.05768}, 2022.

\bibitem{speranca_grove}
L.~D. Speran\c{c}a.
\newblock Totally geodesic {R}iemannian foliations on compact lie groups.
\newblock {\em arXiv:1703.09577}, 2017.

\bibitem{speranca2017on}
L.~D. Speran{\c{c}}a.
\newblock On {Riemannian Foliations} over {Positively} curved manifolds.
\newblock {\em J. Geom. Anal.}, 28(3):2206--2224, Jul 2018.

\bibitem{solorzano}
C. Searle, P. Sol{\'o}rzano, and F. Wilhelm.
\newblock Regularization via {Cheeger} deformations.
\newblock {\em Ann. Global Anal. Geom.}, 48(4):295--303, Dec 2015.

\bibitem{weinstein1980fat}
A. Weinstein.
\newblock Fat bundles and symplectic manifolds.
\newblock {\em Adv. Math.}, 37(3):239--250, 1980.

\bibitem{wilkilng-dual}
B. Wilking.
\newblock A duality theorem for Riemannian foliations in non-negative sectional curvature.
\newblock {\em Geom. Funct. Anal.}, 17:1297--1320, 2007.

\bibitem{wilking2015revisiting}
B. Wilking and W. Ziller.
\newblock Revisiting homogeneous spaces with positive curvature.
\newblock {\em J. reine angew. Math. (Crelles Journal)}, 2015.

\bibitem{Ziller_fatnessrevisited}
W. Ziller.
\newblock Fatness revisited.
\newblock Available at https://www2.math.upenn.edu/~wziller/papers/Fat-09.pdf, 2000.

\bibitem{mutterz}
W. Ziller.
\newblock On {M}. {M}\"uter’s {P}h{D} thesis. 
\newblock Available at https://www.math.upenn.edu/~wziller/papers/SummaryMueter.pdf, 2006.

\bibitem{Zillerpositive}
W. Ziller.
\newblock Riemannian manifolds with positive sectional curvature.
\newblock {\em Lecture Notes in Math.}, 2110, 2014.

\end{thebibliography}
\end{document}